\documentclass[11pt,tbtags]{amsart}
\usepackage{amssymb}
\usepackage{amsmath}
\usepackage{amsthm}
\usepackage[swedish, english]{babel}
\usepackage[latin1]{inputenc}
\usepackage{psfrag}   
\usepackage{graphicx}
\usepackage[square, comma, numbers, sort&compress]{natbib}

\theoremstyle{plain}
\newcommand{\E}{\mathbb E}
\newcommand{\R}{\mathbb R}

\newcommand{\ep}{\epsilon}
\def\P{{\mathbb P}}

\newenvironment{remark}[1][Remark]{\begin{trivlist}
\item[\hskip \labelsep {\bf Remark}]}{\end{trivlist}}
\newtheorem{theorem}{Theorem}[section]
\newtheorem{lemma}[theorem]{Lemma}
\newtheorem{corollary}[theorem]{Corollary}
\newtheorem{proposition}[theorem]{Proposition}
\newtheorem{definition}[theorem]{Definition}

\theoremstyle{definition}

\newenvironment{romenumerate}[1][0pt]{
\addtolength{\leftmargini}{#1}\begin{enumerate}
 }{\end{enumerate}}

\newcommand\ntoo{\ensuremath{{n\to\infty}}}
\newcommand\set[1]{\ensuremath{\{#1\}}}
\newcommand\bigset[1]{\ensuremath{\bigl\{#1\bigr\}}}
\newcommand\bigpar[1]{\bigl(#1\bigr)}
\newcommand\Bigpar[1]{\Bigl(#1\Bigr)}

\newcommand\taux{\tau'}
\newcommand\zst{Z_{s,t}}
\newcommand\eps{\epsilon}
\newcommand\etta{\boldsymbol1}
\newcommand\bbt{\bar b^t}
\newcommand\ggtx{\gamma_{t,x}}
\newcommand\ggtxn{\gamma_{t,x,n}}
\newcommand\downto{\searrow}
\newcommand\upto{\nearrow}

\title{The inverse first-passage problem and optimal stopping}
\author[Erik Ekstr\"om and Svante Janson]{Erik Ekstr\"om$^{1}$ and Svante Janson$^2$}
\address{Uppsala University, Box 480, SE-75106 Uppsala, Sweden.}
\date{31 August, 2015}
\thanks{$^1$ Partly supported by the Swedish Research Council.}
\thanks{$^2$ Partly supported by the Knut and Alice Wallenberg Foundation.}

\begin{document}

\begin{abstract} 
Given a survival distribution on the positive half-axis and a Brownian motion, a solution of the inverse first-passage 
problem consists of
a boundary so that the first passage time over the boundary has the given distribution. We show that
the solution of the inverse first-passage problem coincides with the solution of a related optimal stopping problem.
Consequently, methods from optimal stopping theory may be applied in the study of the inverse first-passage problem.
We illustrate this with a study of the associated integral equation for the boundary.
\end{abstract}

\maketitle

\section{Introduction}

In the inverse first-passage problem one asks for a boundary so that the first passage
time of a Brownian motion over this boundary has a given distribution. 
Anulova \cite{A}
provided the existence of a barrier solution, i.e.
a lower semi-continuous boundary solving the inverse first-passage problem.
The key step in \cite{A} is the approximation of the given distribution 
with discrete ones. 
While the solution for a discrete distribution is straightforward to produce in theory, 
it is less obvious how to determine any qualitative properties of the solution.
Moreover, the convergence of the discrete approximation 
is only guaranteed along a subsequence, and the procedure gives little 
insight about what properties of the discrete solution are preserved in the limit.

A key step towards a better understanding of the inverse first-passage problem was taken 
by Avellaneda and Zhu \cite{AZ}, who argued that the pair consisting of the boundary and 
the distribution of the Brownian paths that did not yet hit the boundary
satisfies a parabolic free-boundary problem. Cheng, Chen, Chadam and Saunders \cite{CCCS1} proved the existence and uniqueness
of solutions (in the viscosity sense) to the corresponding variational inequality, and 
under the assumption that the given survival distribution function is continuous, the same authors showed in 
\cite{CCCS2} that
the solution of the variational inequality indeed provides a solution of the inverse first-passage problem.

Interestingly, the variational inequality suggested in \cite{AZ}, 
and treated mathematically in \cite{CCCS1} and \cite{CCCS2}, is identical to the variational inequality 
associated with a certain optimal stopping problem, thus suggesting a connection between 
inverse first-passage problems and optimal stopping theory.
A main contribution of the current article is to establish this connection rigorously. 
One possible method to accomplish this would be to show that the value function
in the associated optimal stopping problem is a viscosity solution of the corresponding variational 
inequality, and then to appeal to the uniqueness of solutions and the results of \cite{CCCS2}. 
The method presented in the current paper is more direct since we first provide the connection 
in a discrete version of the problem, and then use simple convergence arguments to show that
the connection is preserved in the limit. In this way, we circumvent the use of variational inequalities
and viscosity solutions thereof. Moreover, we are able to prove the connection under minimal regularity 
assumptions on the given distribution, not assuming that it is continuous. We also show uniqueness of the barrier in this generality, 
thus extending the uniqueness results of \cite{CCCS2}.

Apart from the theoretical interest in two seemingly unrelated problems with the same solution, the connection
between the inverse first-passage problem and the related optimal stopping problem enables the use of 
powerful methods from optimal stopping theory also in the study of inverse first-passage problems.
As an illustration of this, we use a mix of techniques 
from first-passage time problems and from optimal stopping theory to study a related integral equation for the
boundary. More precisely, to see that the boundary satisfies an integral equation, we adapt 
arguments from first-passage time problems, and to prove uniqueness of solutions we follow an approach 
from optimal stopping theory.

While all results are formulated for the case of a standard Brownian motion, we do not anticipate any problems
in extending our results to more general diffusion processes as in \cite{CCCS2}. However, for simplicity of the presentation
we refrain from doing this.

\section{Assumptions and  main results}

\begin{definition}
A set $B\subseteq[0,\infty)\times\R$ is a {\em barrier} if
\begin{itemize}
\item[(i)]
$B$ is closed,
\item[(ii)]
$(t,x)\in B$ implies that $(t,y)\in B$ for all $y\geq x$.
\end{itemize}
\end{definition}

If $B$ is a barrier, then one naturally associates with it a lower semi-continuous boundary 
function $b:[0,\infty)\to[-\infty,\infty]$ defined by
\begin{equation}
\label{b}
b(t):=\inf\{x\in\R:(t,x)\in B\}
\end{equation}
(throughout this article we use the convention $\inf\emptyset=\infty$).
Conversely, for a given lower semi-continuous boundary $b:[0,\infty)\to[-\infty,\infty]$
there is an associated barrier defined by
\[B=\{(t,x)\in[0,\infty)\times\R:x\geq b(t)\}.\]
In this way there is a one-to-one correspondence between barriers and lower semi-continuous functions
taking values in the extended real line.

\begin{definition}
\label{survivalfunction}
A function $g:[0,\infty)\to[0,1]$ is a {\em survival distribution function} if
$g$ is non-increasing, right-continuous and satisfies $g(0)=1$ and
  $g(\infty)\geq 0$. 
In other words, 
$1-g$ is a (possibly defective) distribution function.
\end{definition}

Given a standard Brownian motion $W$ with $W_0=0$ and a barrier $B$, define 
the first hitting time $\tau_B$ of $B$ as
\begin{align}
\label{tauB}
\tau_B &:= \inf\{t>0:(t,W_t)\in B\}\\
\notag
&\phantom:= \inf\{t>0: W_t\geq b(t)\}.  
\end{align}
Since $B$ is closed, $\tau_B$ is a stopping time, and
$(\tau_B,W_{\tau_B})\in B$. 
Moreover, 
if $\tau_B>0$ a.s.\ for a given barrier $B$, then the function
\begin{equation}\label{g}
g(t):=\P(\tau_B>t)
\end{equation}
is a survival distribution. 
By the 0-1-law, we have either $\tau_B=0$ a.s.\ or $\tau_B>0$ a.s.
In the {\em first-passage problem}, one aims at determining the survival distribution function $h$ for a given barrier $B$. 
Conversely, in the {\em inverse first-passage problem}, a survival distribution $g$ is given, and one instead seeks a barrier $B$.

We next introduce the associated optimal stopping problem. 
Given a survival distribution function $g$ (as in Definition~\ref{survivalfunction}), let 
the function $v:[0,\infty)\times\R\to[0,1]$ be defined by
\begin{equation}
\label{v}
v(t,x):=\inf_{\gamma\in\mathcal T[0,t]}\E\left[ g(t-\gamma)\etta_{\{\gamma<t\}}+ \etta_{\{\gamma=t, x+W_t\geq 0\}}\right],
\end{equation}
where $\mathcal T[0,t]$ denotes the set of stopping times of the Brownian motion $W$ taking values in $[0,t]$.

\begin{remark}
Note that we start at time 0 and face an optimal stopping problem with horizon $t\geq 0$.
Since the pay-off function $g(t-\cdot)$ is non-decreasing, the only reason not to stop immediately is
the possibility that $x+W$ ends up below 0 at time t.
\end{remark}

Clearly, the value function $v$ is non-decreasing in $x$. Moreover, the non-negativity of the pay-off and the possibility 
to choose $\gamma=t$ yield 
\[0\leq v(t,x)\leq \P\left(x+W_t\geq 0\right).\]
Similarly, the possibility to choose $\gamma=0$ and the monotonicity of $g$ yield 
\begin{eqnarray*}
g(t)\geq v(t,x) &\geq& \inf_{\gamma\in\mathcal T[0,t]}\E\left[g(t)\etta_{\{\gamma<t\}}+ 
g(t)\etta_{\{\gamma=t, x+W_t\geq 0\}}\right]\\
&=& g(t)\P\left( x+W_t\geq 0\right).
\end{eqnarray*}
Consequently, the value function $v$ satisfies 
\[\lim_{x\to-\infty}v(t,x)=0\mbox{ and }\lim_{x\to\infty}v(t,x)=g(t)\] 
for each fixed $t$.

We now present our main results. The first one states that the value function $v$ of the optimal 
stopping problem provides a solution to the inverse first-passage problem.

\begin{theorem}
\label{main}
{\bf (Solution of the inverse first-passage problem.)}
Let a survival distribution function $g$ be given,
and let $v$ be the value of the optimal stopping problem
defined in \eqref{v}. Then 
\begin{equation}
\label{Bfromv}
B:=\{(t,x)\in[0,\infty)\times\R:v(t,x)=g(t)\}
\end{equation}
is a barrier, and $P(\tau_B>t)=g(t)$. 
\end{theorem}

Furthermore, the solution in Theorem~\ref{main} is unique up to some trivial modifications that
we eliminate by the next definition.

\begin{definition}
We say that a barrier $B$ is {\em standard} if the corresponding boundary $b$ satisfies 
\begin{itemize}
\item[(i)]
$b(0)=0$ 
\item[(ii)]
$b(t)=-\infty \mbox{ for some }t>0\implies b(s)=-\infty$ for all $s>t$.
\end{itemize}
\end{definition}

\begin{remark}
If $B$ is a barrier with $\tau_B>0$ a.s., then there exists a standard barrier $\overline B$ such that
$\tau_B=\tau_{\overline B}$. The barrier $\overline B$ can be defined via its boundary function $\overline b$ by
\begin{equation}\label{barb}
\overline b(t)=
\begin{cases}
0, & t=0\\
b(t), &  0<t<T_B\\
-\infty, & t\geq T_B,  
\end{cases}  
\end{equation}
where $T_B:=\inf \{t>0:b(t)=-\infty\}$. Indeed, first note that the value of $b(0)$ does not influence the
corresponding hitting time (since the infimum in \eqref{tauB} is taken over strictly positive times).
Moreover, since $\tau_B>0$ we must have $\liminf_{t\downto 0}b(t)\geq 0$, 
so defining $\overline b(0)=0$ does not destroy the lower semi-continuity of the boundary. Similarly, 
defining $\overline b(s)=-\infty$ for $s>T_B$ influences neither the corresponding 
hitting time nor the lower semi-continuity of the boundary.
\end{remark}

\begin{theorem}
\label{main2}
{\bf (Uniqueness.)}
Let a survival distribution $g$ be given. Then the barrier $B$ defined in \eqref{Bfromv} is standard,
and it is the unique standard barrier with $\P(\tau_B>t)=g(t)$.
\end{theorem}

Given a barrier $B$, define
\begin{equation}\label{u}
u(t,x):=\P(W_t\leq x, \tau_B>t),
\end{equation}
The proofs of Theorems~\ref{main} and \ref{main2} use the following
result of independent interest. 

\begin{theorem}
\label{main3}
{\bf ($u=v$)}
Let $B$ be a barrier and $g$ the corresponding survival distribution function
$g(t)=\P(\tau_B>t)$. Then the function $u$ defined in \eqref{u} 
and the value 
$v$ of the optimal stopping problem defined in \eqref{v} satisfy $u=v$ on
$[0,\infty)\times\R$. 
\end{theorem}

Another relation between the optimal stopping problem in \eqref{v} and
the (inverse) first-passage problem is that the infimum in \eqref{v} is
attained by a stopping time that is the hitting time of a barrier obtained
by reflecting $B$. To be more precise, fix $t>0$, let $\bbt$ be the
reflected boundary function
\begin{equation}
  \label{bbt}
\bbt(s):=
\begin{cases}
  b(t-s), & 0\le s<t
\\
-\infty, &s\ge t
\end{cases}
\end{equation}
and let 
\begin{equation}\label{ggtx}
  \ggtx:=\inf\bigset{s\ge0:x+W_s\ge \bbt(s)}.  
\end{equation}
In other words, 
if we regard $u\mapsto x+W_{t-u}$ as a backward Brownian motion with
``time'' $u$ decreasing from $t$ to 0, starting at $(t,x)$, then $t-\ggtx$
is the time this backward Brownian motion hits $B\cup\set{t=0}$.
Note that $\ggtx$ is a stopping time with $0\le\ggtx\le t$, i.e.,
$\ggtx\in \mathcal T[0,t]$. 

\begin{remark}
Note that we allow $s=0$ in \eqref{ggtx}, unlike in \eqref{tauB}. Thus
$\ggtx=0$ when $x\ge\bbt(0)=b(t)$. Conversely, $\ggtx>0$ when $x<b(t)$  
because $b$ is lower semicontinuous.
\end{remark}

\begin{theorem}
  \label{Tstopp}
The stopping time $\ggtx$ is optimal for the optimal stopping problem
\eqref{v}, i.e., 
\begin{equation}
\label{vopt}
v(t,x)=\E\left[ g(t-\ggtx)\etta_{\{\ggtx<t\}}+ \etta_{\{\ggtx=t, x+W_t\geq 0\}}\right].
\end{equation}
\end{theorem}

The proofs of Theorems~\ref{main}, \ref{main2}, \ref{main3} and \ref{Tstopp} are
carried out in Sections~\ref{auxiliary}--\ref{S:stopp} 
below. 
Moreover, in Section~\ref{regularity} we show, under additional assumptions on the survival distribution,
that the solution of the inverse first-passage problem is continuous, and in
Section~\ref{S:inteqn} that it is characterised as the unique solution of an integral equation.

\section{An auxiliary result}
\label{auxiliary}

Let $B$ be a given barrier, and let $\tau_B$ be the corresponding first
hitting time. 
Define as in \eqref{g} and \eqref{u} 
the functions $g:[0,\infty)\to[0,1]$ 
and $u:[0,\infty)\times\R\to[0,1]$ by 
$g(t):=\P(\tau_B>t)$
and $u(t,x)=\P(W_t\leq x, \tau_B>t)$.
Then $g$ is either a survival distribution function or $g=0$ (which happens if $\tau_B=0$), 
the function $u$ is non-decreasing in $x$ with $0\leq u(t,x)\leq g(t)$, and
\[\lim_{x\to-\infty}u(t,x)=0\mbox{ and }\lim_{x\to\infty}u(t,x)=\P(\tau_B>t)=g(t).\] 

\begin{remark}
Since $\P(W_t=x)=0$ for any $(t,x)\in(0,\infty)\times\R$, we have 
\[u(t,x)=\P(W_t\leq x, \tau_B>t)=\P(W_t< x, \tau_B>t)\]
for any $t>0$. In particular, 
\[u(t,b(t))=\P(W_t< b(t), \tau_B>t)=\P(\tau_B>t)=g(t).\]
\end{remark}

The following result shows that a standard barrier $B$ can be 
recovered from $u$.

\begin{proposition}
\label{recoveryfromu}
Assume that $B$ is a standard barrier with $\tau_B>0$ a.s., 
and let $g$ and $u$ be defined by
\eqref{g} and \eqref{u}, respectively.
Then 
\begin{equation}
  \label{Bfromu}
B=\{(t,x)\in [0,\infty)\times\R :u(t,x)=g(t)\}.
\end{equation}
\end{proposition}

\begin{proof}
Define
\[D:=\{(t,x)\in [0,\infty)\times\R :u(t,x)=g(t)\}.\]
We thus claim $B=D$.

First note that $u(0,x)=\P(0\leq x, \tau_B>0)=\etta_{[0,\infty)}(x)$ and $g(0)=1$, 
so $u(0,x)=g(0)$ if and only if $x\geq 0=b(0)$. Consequently, $B\cap \{t=0\}=D\cap\{t=0\}$.
Similarly, if $t\geq T_B$ then $u(t,x)=g(t)=0$, so $u(t,x)=g(t)$ for $x>-\infty=b(t)$. 
Thus $B\cap \{t\geq T_B\}=D\cap \{t\geq T_B\}$.
Next, if $(t,x)\in B$ with $t>0$, then 
\[u(t,x) = \P(W_t\leq x, \tau_B>t)=
\P(W_t\leq b(t), \tau_B>t)=\P(\tau_B>t)=g(t),\]
so
\[B\cap\{t>0\}\subseteq D\cap\{t>0\}.\]

It remains to show that 
\begin{equation}
\label{BequalsD} 
B\cap\{0<t<T_B\}\supseteq D\cap\{0<t<T_B\}.
\end{equation}
To do that, assume that $t\in(0,T_B)$ and $(t,x)\notin B$. Since $\tau_B>0$ a.s., there exists 
a time point $t_0\in(0,t)$ and an $\ep>0$ such that the probability 
of the event $\{\tau_B>t_0, W_{t_0}\in (b(t_0)-\ep,b(t_0)-2\ep)\}$ is strictly positive.
Since $b$ is lower semi-continuous, one can find two continuous functions $a_1(t)$ and $a_2(t)$ on $[t_0,t]$ 
such that 
\begin{itemize}
\item
$(a_1(t_0),a_2(t_0))\supseteq (b(t_0)-\ep/2,b(t_0)-5\ep/2)$,
\item
$(a_1(t),a_2(t))\subseteq (x,b(t))$,
\item
$\min_{s\in[t_0,t]} (a_2(s)-a_1(s))> 0$, 
\item
$\min_{s\in[t_0,t]}(b(s)-a_2(s))>0$.
\end{itemize}
It follows that 
\begin{eqnarray*}
\P(W_{t}\in (x,b(t)),\tau_B>t) &\geq& \P(W_{s}\in (a_1(s),a_2(s))\forall s\in[t_0,t],\tau_B>t_0)\\
&>& 0.
\end{eqnarray*}
Consequently, 
\[u(t,x)=\P(W_t\leq x,\tau_B>t)< \P(W_{t}\leq b(t),\tau_B>t)=g(t),\]
so $(t,x)\notin D$. Thus \eqref{BequalsD} holds, which finishes the proof.
\end{proof}

\section{The discrete problems}

In this section we prove Theorems~\ref{main} and \ref{main3} in the case of a given {\em discrete} survival distribution.
To do this, assume that we are given a survival distribution function $g:[0,\infty)\to[0,1]$
which is piecewise constant outside a finite set of 
distinct points $\{t_k\}_{k=1}^N$, where 
\[0<t_1<t_{2}<\dots<t_N<\infty.\] 
Setting $t_0=0$ and $t_{N+1}=\infty$, by right-continuity, $g$ is constant on each interval 
$[t_k,t_{k+1})$, $k \in\{0,1,\dots,N\}$. 
Since $g$ is non-increasing, we necessarily have $g(t_k)-g(t_{k+1})\geq 0$, but we
do not assume strict inequality.

\subsection{The discrete inverse first passage problem}
\label{discrete-inverse}

Let a piecewise constant survival distribution function $g$ be given. We construct a barrier $B$ (equivalently, a boundary $b$) 
and its corresponding distribution function 
\[u(t,x)=\P(W_t\leq x, \tau_B>t)\]
recursively as follows.

\begin{itemize}
\item
Set $b(t_0)=0$, and let $b(t)=\infty$ for $t\in(t_0,t_1)$.
Then $\tau_B$ has no mass in $[t_0,t_1)$, so $u(t,x)=\P(W_t\leq x)$ for $t<t_1$. 
Note that $x\mapsto \P(W_{t_1}\leq x)$ is strictly increasing and maps $[-\infty,\infty]$ onto $[0,g(t_0)]=[0,1]$.
Let $b(t_1)\in[-\infty,\infty]$ be the unique value such that $\P(W_{t_1}\leq b(t_1))=g(t_1)$.
If $b(t_1)=-\infty$, set $b(t)=-\infty$ also for $t\in(t_1,\infty)$, which finishes the construction. 
Otherwise, if $b(t_1)>-\infty$, then set $b(t)=\infty$ for $t\in(t_1,t_2)$.
In this way, $\P(\tau_B>t)=g(t)$ on $[0,t_2)$.
\item
Assume that $b(t_i)>-\infty$, $i=1,\dots,k$, that $b(t)=\infty$ for $t\in(0,t_{k+1})\setminus\{t_1,\dots,t_k\}$, and that 
$\P(\tau_B>t)=g(t)$, $t\in(0,t_{k+1})$.
Then $u(t,x)=\P(W_t\leq x, \tau_B>t)=\P(W_t\leq x, \tau_B>t_k)$
for $t\in(t_k,t_{k+1})$, and $x\mapsto\P(W_{t_{k+1}}\leq x, \tau_B>t_k)$ is strictly  increasing, mapping $[-\infty,\infty]$ onto 
$[0,g(t_k)]$. Let $b(t_{k+1})\in[-\infty,\infty]$ be the unique value such that 
$\P(W_{t_{k+1}}\leq b(t_{k+1}), \tau_B>t_k)=g(t_{k+1})$.
If $b(t_{k+1})=-\infty$, then we put $b(t)=-\infty$ for all $t>t_{k+1}$, and the construction is finished.
Otherwise, set $b(t)=\infty$ for $t\in (t_{k+1},t_{k+2})$, and then this second step is repeated for $k+1$ instead of $k$.
\end{itemize}
The above construction is terminated either if $b(t_k)=-\infty$ for some $t_k$ or when $b(t_k)$ has been determined for
all $k=1,2,\dots,N$. By construction, the barrier $B$ solves the discrete inverse first-passage problem, and it is easy to see that
the solution is unique in the class of standard barriers.
We have the following inductive representation of the corresponding distribution function $u$.

\begin{theorem}
\label{forward}
{\bf (An inductive scheme for the discrete inverse first-passage problem.)}
Let $B$ be the barrier constructed above which solves the discrete inverse first-passage problem, and let 
\[u(t,x)=\P(W_t\leq x,\tau_B>t).\]
Then $u$ satisfies the following inductive scheme.
\begin{itemize}
\item
For $t=t_0=0$,
\begin{equation}
\label{initialcond}
u(0,x)=
\begin{cases}
1, & x\geq  0\\
0, & x< 0.  
\end{cases}
\end{equation}
\item
On each strip $(t_k,t_{k+1})\times \R$, $k\geq 0$, we have
\begin{equation}
\label{forwardeq}
u(t,x)=  \E u( t_k, x+W_{t-t_k}).
\end{equation}
\item
For $t=t_{k+1}$, $k\geq 0$,
\begin{equation}
\label{updating}
u(t_{k+1},x)= \min\left\{\E u(t_{k},x+W_{t_{k+1}-t_k}), g(t_{k+1})\right\}.
\end{equation}
\end{itemize}
\end{theorem}

\begin{proof}
The initial condition \eqref{initialcond} is immediate from $W_0=0$ and $\tau_B>0$.
Next, if $t\geq t_k$, let $s:=t-t_k$ and let $W_s^\prime:=W_t-W_{t_k}$; then $W^\prime_s\overset{d}{=}W_s$ and 
$W^\prime_s$ is independent of $(W_r)_{r\leq t_k}$. Thus, by conditioning on $W^\prime_s$,
\begin{eqnarray*}
\P(W_t\leq x,\tau_B>t_k) &=& \P( W_{t_k}\leq x-W^\prime_s,\tau_B>t_k)\\
&=& \E u(t_k,x-W_s^\prime) = \E u(t_k, x-W_{t-t_k}).
\end{eqnarray*}
Hence, if $t\in[t_k,t_{k+1})$, then 
\[u(t,x)=\P(W_t\leq x, \tau_B>t) = \P(W_t\leq x, \tau_B>t_k)
= \E u(t_k, x-W_{t-t_k})\]
which proves \eqref{forwardeq},
since $W_{t-t_k}$ is symmetric.

Furthermore, for $k<N$, $\tau_B> t_{k+1}$ if and only if $\tau_B>t_k$ and $W_{t_{k+1}}< b(t_{k+1})$, and thus,
recalling the definition of $b(t_{k+1})$,
\begin{eqnarray*}
u(t_{k+1},x) &=& \P(W_{t_{k+1}}\leq x\wedge b(t_{k+1}),\tau_B>t_k)\\
&=& \P(W_{t_{k+1}}\leq x, \tau_B>t_k)\wedge g(t_{k+1}) \\
&=& \E u(t_{k},x+W_{t_{k+1}-t_k})\wedge g(t_{k+1}),
\end{eqnarray*}
which proves \eqref{updating}.
\end{proof}

\subsection{The discrete optimal stopping problem}

In this subsection we study the associated optimal stopping problem when the given survival distribution function 
$g$ is piecewise constant. 

\begin{theorem}
\label{backward}
{\bf (An inductive scheme for the discrete optimal stopping problem.)}
Assume that $g$ is a piecewise constant survival distribution function.
Then the optimal stopping value $v$ satisfies the following inductive scheme.

\noindent
\begin{itemize}
\item
For $t= t_0=0$,
\begin{equation}
\label{termcond}
v(0,x)=
\begin{cases}
1, & x\geq  0\\
0, & x< 0.  
\end{cases}
\end{equation}
\item
For $t\in ( t_k, t_{k+1})$, $k\geq 0$,
\begin{equation}
\label{backwind0}
v(t,x)=  \E v( t_k, x+W_{t-t_k}).
\end{equation}
\item
For $t=t_{k+1}$, 
\begin{equation}
\label{backwind}
v(t_{k+1},x)= \min\left\{ \E v( t_k,x+W_{t_{k+1}-t_k}),  g( t_{k+1})\right\}.
\end{equation}
\end{itemize}
\end{theorem}

\begin{proof}
The initial condition \eqref{termcond} follows immediately from the definition of $v$.
For $t\geq 0$, let $\overline{\mathcal T}[0,t]$ be the set of stopping times taking values in 
$\{ t-t_k,k\geq 0\}\cap[0,t]$, and define
\[\overline v(t,x)=\inf_{\gamma\in\overline{\mathcal T}[0,t]}
\E\left[  g(t-\gamma)\etta_{\{\gamma<t\}}+ \etta_{\{\gamma=t, x+W_t\geq 0\}}\right].\]
Then clearly $v\leq\overline v$ since $\overline{\mathcal T}[0,t]\subseteq \mathcal T[0,t]$.
On the other hand, given a stopping time $\gamma\in\mathcal T[0,t]$, let
\[\overline\gamma:=\inf\{s\geq \gamma:s= t-t_k\mbox{ for some }k\ge0\}.\]
Then $\overline\gamma\in\overline{\mathcal T}[0,t]$ and 
\[g(t-\gamma)\etta_{\{\gamma<t\}}+ \etta_{\{\gamma=t, x+W_t\geq 0\}}\geq
g(t-\overline\gamma)\etta_{\{\overline\gamma<t\}}+ \etta_{\{\overline\gamma=t, x+W_t\geq 0\}}\]
(strict inequality may happen if $\gamma$ takes values in $( t- t_1,t)$). This implies that also
$v\geq\overline v$, so $v=\overline v$.

Thus $v$ coincides with the value function $\overline v$ of a discrete time
optimal stopping problem.
It is well-known, and easy to see, that such a value function 
can be determined using backward induction (see for example \cite[Chapter
  I.1]{PS}), \eqref{backwind0} and \eqref{backwind}  
follow.
In fact,  an optimal stopping time
in $\overline{\mathcal T}[0,t]$ is obtained by stopping at $t-t_k\ge0$ if
$v(t_k,x+W_{t-t_k})=g(t_k)$; more precisely, we stop at the first such $t-t_k$,
i.e., the one with the largest $t_k$, and if no such $t_k\le t$ exists, then
we stop at $t$. 
\end{proof}

\begin{corollary}
  \label{Cu=v}
Assume $B$ is a barrier  
such that the survival distribution function $g$ is piecewise constant,
Then $u(t,x)= v(t,x)$ for all $(t,x)\in[0,\infty)\times \R$.
In other words, 
Theorem \ref{main3} holds in the discrete case.
\end{corollary}

\begin{proof}
Suppose first that $B$ is standard. As said above, the solution to the
discrete inverse first-passage problem is unique in the class of standard
barriers, so $B$ equals the barrier constructed from $g$ above.
The result now follows from Theorems~\ref{forward} and \ref{backward}. 

In general, we can replace $B$ by a standard barrier $\overline B$ by
\eqref{barb} without changing $\tau_B$; thus $B$ and $\overline B$ have the
same $u(t,x)$ and $v(t,x)$, and the general result follows.
\end{proof}

\begin{corollary}
\label{CX}
Theorems \ref{main} and \ref{Tstopp} hold in the special case of a piecewise
constant $g$.  
\end{corollary}

\begin{proof}
It follows from Corollary \ref{Cu=v} and Proposition~\ref{recoveryfromu}
that the discrete inverse first passage problem is solved by
\begin{equation}\label{cb}
B=\{(t,x)\in[0,\infty)\times\R:v(t,x)=g(t)\},  
\end{equation}
which verifies Theorem~\ref{main} in this case.

Moreover, an optimal stopping time for \eqref{v} in the discrete case is
given by the rule at the end of the proof of Theorem \ref{backward}.
 This stopping time $\gamma$ is given by the smallest
$s$ of the form $s=t-t_k$ with $t_k\le t$ and $v(t_k,x+W_{t-t_k})=g(t_k)$;
if no such $t_k$ exists, then $\gamma=t$.
However, using \eqref{cb} and \eqref{bbt}, for $t_k\le t$, 
\begin{equation*}
  \begin{split}
v(t_k,x+W_{t-t_k})=g(t_k)	
&\iff
(t_k,x+W_{t-t_k})\in B
\iff
x+W_{t-t_k} \ge b(t_k)
\\&
\iff
x+W_{t-t_k} \ge \bbt(t-t_k).
  \end{split}
\end{equation*}
Hence,
$\gamma=\inf\bigset{s\in\set{t-t_k:k\ge0}\cap[0,t]:x+W_s\ge\bbt(s)}\wedge t$.
Furthermore, $b(s)=\infty$ if $s\notin\set{t_k}$, so $\bbt(s)=\infty$ if
$s\in[0,t)\setminus\set{t-t_k:k\ge1}$, while $\bbt(t)=-\infty$.
Hence, 
$\gamma=\inf\bigset{s\ge0:x+W_s\ge\bbt(s)}=\ggtx$.
\end{proof}

\section{The general case}
\label{general}

In this section we prove Theorem~\ref{main} by approximating the survival distribution function with 
piecewise constant ones. 

To do that, let $g$ be a given survival distribution function, and for each $n\geq 1$ 
let $\mathbb T^n=\{t^n_k,k=0,\dots,N^n+1\}$ be a set such that
\begin{itemize}
\item
$0=t^n_0<t^n_1<\dots<t^n_{N^n}<t^n_{N^n+1}=\infty$;
\item
$g(t^n_{k})-g(t^n_{k+1}-)\leq 1/n$ for $k=0,\dots,N^n$;
\item
$\mathbb T^n\subseteq \mathbb T^{n+1}$.
\end{itemize}
It is easily seen that such sets exist. Moreover, 
note that all times $t$ such that $g(t-)-g(t)>1/n$ are necessarily contained in $\mathbb T^n$; hence all jump 
times of $g$ are contained in $\cup_{n=1}^\infty \mathbb T^n$.
Define piecewise constant survival distribution functions $g^n$ by 
\[g^n(t):=g(t_k^n)\mbox{ for }t\in[t_k^n,t_{k+1}^n),\]
and note that
\begin{equation}
\label{gn}
g^{n}(t)-1/n\leq g(t)\leq g^n(t)
\end{equation}
for all $t$. Let $B^n$ be the barrier for $g^n$ constructed in Section~\ref{discrete-inverse}, let 
\[\tau_{B^n}= \inf\{t>0:(W_t,t)\in B^n\}\]
be the corresponding hitting time so that 
\[\P(\tau_{B^n}>t)=g^n(t)\]
for $t\geq 0$, and let 
\[u^n(t,x)=\P(W_t\leq x,\tau_{B^n}>t).\]

We apply the following theorem from \cite{A}.

\begin{theorem}
\label{anulova}
There exists a subsequence $n_m$ and a standard barrier $B$ such that
\begin{itemize}
\item[(i)]
$B$ solves the inverse first passage problem, i.e. $\P(\tau_B>t)=g(t)$;
\item[(ii)]
$\tau_{B^{n_m}}\to \tau_B$  a.s.\ as $m\to\infty$.
\end{itemize}
\end{theorem}

\begin{remark}
In \cite{A}, Anulova considered two-sided (symmetric) barriers of the type
\[B=\{(t,x)\in [0,\infty)\times\R:\vert x\vert\geq b(t)\},\]
where $b\geq 0$. It is straightforward to 
check that the results of \cite{A} carry over also to our setting with one-sided barriers. 
Statement (i) in Theorem~\ref{anulova} is the main result of \cite{A} (we use a different discretisation of the time axis than Anulova does, 
but this does not influence her result).
Moreover, Anulova proves the existence of a subsequence along which the corresponding hitting times
converge in probability. Since any sequence that converges in probability has a subsequence that converges almost surely, 
this yields (ii).
\end{remark}

Let $u(t,x)=\P(W_t\leq x,\tau_{B}>t)$, where $B$ is the barrier in Theorem~\ref{anulova}.

\begin{corollary}
\label{conv}
We have 
$u^{n_m}(t,x)\to u(t,x)$ as $m\to\infty$ for $(t,x)\in[0,\infty)\times\R$.
\end{corollary}

\begin{proof}
Let $(t,x)\in[0,\infty)\times\R$.
By the almost sure convergence of $\tau_{B^{n_m}}$ to $\tau_B$ we have that
\[\liminf_{m\to\infty} \{\tau_{B^{n_m}}>t\}\supseteq \{\tau_{B}>t\},\]
and thus 
\[\P(\tau_B>t\geq \tau_{B^{n_m}})\to 0.\] 
Furthermore, using \eqref{gn},
$\P(\tau_{B^{n_m}}>t)=g^{n_m}(t)\to g(t)=\P(\tau_{B}>t)$ as $m\to\infty$.
Hence
\[\P(\tau_{B^{n_m}}>t\geq \tau_B)=\P(\tau_{B^{n_m}}>t) + \P(\tau_B>t\geq \tau_{B^{n_m}})-\P(\tau_B>t)\to 0\]
as $m\to\infty$.
Consequently,
\begin{eqnarray*}
\vert u^{n_m}(t,x)-u(t,x)\vert &\leq& \P(\tau_{B^{n_m}}>t\geq \tau_B)+ \P(\tau_{B}>t\geq\tau_{B^{n_m}}) \to 0
\end{eqnarray*}
as $m\to\infty$, which finishes the proof.
\end{proof}

We next show that Theorem \ref{main3} holds for the barriers constructed by
the procedure above.

\begin{corollary}\label{Cu=v2}
  Let $B$ be the barrier in Theorem \ref{anulova}. Then
$u(t,x)=v(t,x)$ 
for all $(t,x)\in[0,\infty)\times\R$.
\end{corollary}

\begin{proof}
Let
\[v(t,x)=\inf_{\gamma\in\mathcal T[0,t]}\E\left[g(t-\gamma)\etta_{\{\gamma<t\}}+ \etta_{\{\gamma=t, x+W_t\geq 0\}}\right]\]
and
\[v^n(t,x)=\inf_{\gamma\in\mathcal T[0,t]}\E\left[g^n(t-\gamma)\etta_{\{\gamma<t\}}+ \etta_{\{\gamma=t, x+W_t\geq 0\}}\right]\]
be the value functions of the corresponding optimal stopping problems.
By Corollary \ref{Cu=v},
$u^n(t,x)=v^n(t,x)$
for all $(t,x)\in[0,\infty)\times\R$.
Moreover, \eqref{gn} yields that $v^n-1/n\leq v\leq v^n$, so $v^n\to v$
as $n\to\infty$, and, along the subsequence $n_m$, 
$u^{n_m}\to u$ by Corollary~\ref{conv}. Hence $u=v$.
\end{proof}

\begin{proof}[Proof of Theorem~\ref{main}]
The barrier $B$ in Theorem \ref{anulova} satisfies $\P(\tau_B>t)=g(t)$.
Furthermore,
it follows from Proposition~\ref{recoveryfromu} that
$B$ can be recovered from $u$ and $g$ by \eqref{Bfromu}, 
and thus \eqref{Bfromv} follows
by Corollary \ref{Cu=v2}.
\end{proof}

\section{Uniqueness of solutions to the inverse first-passage problem}
\label{S:uniqueness}

Theorem~\ref{main} shows that a solution of the inverse first-passage problem is provided by 
the associated optimal stopping problem. 
In this section we study uniqueness of solutions. In \cite{CCCS2}, a uniqueness result is proved for the class of continuous
survival distributions using viscosity solutions of the associated variational inequality. Our proof is partly adapted from
their analysis, but using the connection with the associated optimal stopping problem allows us to extend the
uniqueness result to arbitrary survival distributions.

For a given standard barrier $B$ with boundary $b$, let, as in \cite{CCCS2}, 
\[t_k^n=\inf\left\{t\in\left[\frac{k}{2^n},\frac{k+1}{2^n}\right]:b(t)=\inf_{s\in[\frac{k}{2^n},\frac{k+1}{2^n}]}b(s)\right\}\]
for $k,n\geq 1$, and define 
\[A^1_n(b)=\{t_k^n:k=1,2,\dots,n2^n\}.\]
Moreover, let $t_1,t_2,\dots$ be an enumeration of the set $\{t:\P(\tau_B=t)>0\}$, let
\[A^2_n(b)=\{t_1,t_2,\dots,t_n\},\]
and set $A_n(b)=A^1_n(b)\cup A^2_n(b)$.
Then $A_n(b)\subseteq A_{n+1}(b)$, and we define $A(b)=\cup_{n=1}^\infty
A_n(b)$. 
Define the barrier $B_n:=\set{(t,x)\in B:t\in A_n(b)}$
and  the corresponding stopping time 
\begin{equation}
  \label{taun}
\tau_n:=\tau_{B_n}=\min\{t\in A_n(b):W_t\geq b(t)\}.
\end{equation}
Then $\tau_n\geq \tau_{n+1}\geq \tau_B$.
\begin{lemma}
  \label{LP4}
$\lim_\ntoo \tau_n = \tau_B$ a.s.
\end{lemma}

This  follows from \cite[Propositions 2 and 4]{CCCS2}, but for completeness,
we give a simplified version of the proof given 
(for a more general situation)
there.
We define the stopping time
\begin{equation}
  \taux_B:=\inf\set{t>0:W_t>b(t)}.
\end{equation}
Obviously, $\taux_B\ge\tau_B$.
We first show
the following result, saying that  equality holds, which is part of 
\cite[Proposition 2]{CCCS2}.
\begin{lemma}[\cite{CCCS2}]
  \label{LP2}
$\taux_B=\tau_B$ a.s.
\end{lemma}
\begin{proof}
  Let $0<s<t$ and let $\zst:=\sup_{s\le u\le t}(W_u-b(u))$.
Then
\begin{equation*}
  \zst=W_s+\sup_{s\le u\le t}(W_u-W_s-b(u)),
\end{equation*}
where the two terms on the right-hand side are independent and $W_s$ has a
continuous distribution. Hence $\zst$ has a continuous distribution and
$\P(\zst=0)=0$.

Now suppose that $\tau_B\in[s,t]$. Then $\zst\ge W_{\tau_B}-b(\tau_B)\ge0$ and
thus 
a.s.\ $\zst>0$, which means that there exists $u\in[s,t]$ such that
$W_u-b(u)>0$ and thus $\taux_B\le u$. Since $\taux_B\ge\tau_B$, we then have
$s\le\tau_B\le\taux_B\le t$.

We have shown that if $0<s<t$, then a.s.
\begin{equation}\label{tau=}
  \tau_B\in[s,t]\implies \taux_B\in[s,t].
\end{equation}
Hence, a.s.\ \eqref{tau=} holds for all rational $s$ and $t$ with $0<s<t$,
which implies $\tau_B=\taux_B$.
\end{proof}

\begin{proof}[Proof of Lemma \ref{LP4}]
  As said above, $\tau_n\ge\tau_B$, so it suffices to show that
$\limsup_\ntoo\tau_n\le \tau_B$.
By Lemma \ref{LP2}, we may assume $\tau_B=\taux_B$, and since the result is
trivial when $\tau_B=\infty$, we may assume $\tau_B<\infty$.

Let $\eps>0$.
Since $\taux_B=\tau_B$, there exists $t<\tau_B+\eps$ such that $W_t>b(t)$.
By the continuity of $W$, there exists $\delta>0$ such that 
$W_u>b(t)$ if $|u-t|< \delta$.
Now assume that $n$ is so large that $2^{-n}<\delta$ and 
$2^{-n}<t<n$ and
let $k:=\lfloor 2^nt\rfloor$, so $t\in[k2^{-n},(k+1)2^{-n}]$.
Then $|t-t_k^n|\le 2^{-n}< \delta$ so $W_{t_k^n}>b(t)\ge b(t_k^n)$;
furthermore $t_k^n\in A_n^1(b)\subseteq A_n(b)$,  and thus by \eqref{taun},
$\tau_n\le t_k^n\le t+2^{-n}\le \tau_B+\eps+2^{-n}$.
Hence, $\tau_n\le\tau_B+2\eps$ for all large $n$.
\end{proof}

So far we have not needed the points in $A_n^2(b)$; however, they are
essential for the next result.
\begin{lemma}\label{Lg}
  As \ntoo, for every $t\ge 0$,
  \begin{equation}\label{gng}
	g_n(t):=\P(\tau_n>t)\to \P(\tau_B>t)=: g(t) .
  \end{equation}
\end{lemma}

\begin{proof}
  First consider a fixed $t\ge0$.
Since $\tau_n\ge\tau_{n+1}\ge\tau_B$, we have 
$g_n(t)\ge g_{n+1}(t)\ge g(t)$. In particular, $\lim_\ntoo g_n(t)$ exists.
Moreover, by Lemma \ref{LP4}, $\tau_n\to\tau$ a.s., and thus in distribution, 
whence
\begin{equation}
  g(t)=\P(\tau_B>t)\le\lim_\ntoo g_n(t)\le\P(\tau_B\ge t)=g(t-).
\end{equation}
In particular, \eqref{gng} holds for every $t$ such that $\P(\tau_B=t)=0$
and thus $g(t)=g(t-)$.

On the other hand, if $\P(\tau_B=t)>0$, then $t\in A_n^2(b)\subseteq A_n(b)$
for all 
sufficiently large $n$, and then $\tau_B=t$ implies $\tau_n= t$ by
\eqref{taun}. Consequently, if $\tau_n>t$ for all $n$, then $\tau_B\neq t$,
and thus a.s., $\tau_B=\lim_\ntoo\tau_n>t$.
Consequently,
\begin{equation}
\lim_\ntoo g_n(t)=
  \lim_\ntoo \P(\tau_n>t)
= \P\Bigpar{\bigcap_n \set{\tau_n>t}}
=\P(\tau_B>t)=g(t)
\end{equation}
and \eqref{gng} holds in this case too.
%
\end{proof}

\begin{remark}
  Using the sets $\mathbb T^n$ defined in Section \ref{general}, it is easy
  to show that, moreover,
$g_n(t)\to g(t)$ uniformly for all $t\ge0$.
We omit the details since we do not use this property. 
\end{remark}

\begin{proof}[Proof of Theorem \ref{main3}]
Let $u_n$ and $v_n$ be the functions given by \eqref{u} and \eqref{v} for
the barrier $B_n$.
Since $B_n$ solves the inverse first-passage problem for the piecewise
constant survival distribution function $g_n$, Corollary \ref{Cu=v} shows that
$u_n(t,x)=v_n(t,x)$ for all $(t,x)\in[0,\infty)\times \R$.

Moreover, for any $(t,x)\in[0,\infty)\times \R$,
	\begin{equation*}
	  \begin{split}
0\le u_n(t,x)-u(t,x) 
&=\P(W_t\leq x, \tau_B\le t<\tau_n)
\\
&\le \P(\tau_B\le t<\tau_n)=g_n(t)-g(t). 		
	  \end{split}
	\end{equation*}
Hence, Lemma \ref{Lg} implies that $u_n(t,x)\to u(t,x)$ as \ntoo.

Furthermore, 
since $g_n\ge g$, \eqref{v} yields $v_n(t,x)\ge v(t,x)$,
and thus
\begin{equation}
  \label{liminfvn}
\liminf_\ntoo v_n(t,x) \ge v(t,x).
\end{equation}
Next,
suppose that $\gamma\in\mathcal T[0,t]$.
Then
\begin{equation*}
 g_n(t-\gamma)\etta_{\{\gamma<t\}}+ \etta_{\{\gamma=t, x+W_t\geq
	0\}}
\to
 g(t-\gamma)\etta_{\{\gamma<t\}}+ \etta_{\{\gamma=t, x+W_t\geq 0\}}
\end{equation*}
a.s.\ by Lemma \ref{Lg}, and thus by dominated convergence
\begin{equation*}
\E\left[ g_n(t-\gamma)\etta_{\{\gamma<t\}}+ \etta_{\{\gamma=t, x+W_t\geq
	0\}}\right]
\to
\E\left[ g(t-\gamma)\etta_{\{\gamma<t\}}+ \etta_{\{\gamma=t, x+W_t\geq 0\}}\right]
\end{equation*}
The left-hand side is, by the definition \eqref{v}, at least $v_n(t,x)$, and thus
\begin{equation*}
\limsup_\ntoo v_n(t,x)
\le
\E\left[ g(t-\gamma)\etta_{\{\gamma<t\}}+ \etta_{\{\gamma=t, x+W_t\geq 0\}}\right].
\end{equation*}
Taking the infimum over all 
$\gamma\in\mathcal T[0,t]$, we obtain
\begin{equation}\label{limsupvn}
\limsup_\ntoo v_n(t,x)
\le v(t,x).
\end{equation}
Together, \eqref{liminfvn} and \eqref{limsupvn} yield
$v_n(t,x)\to v(t,x)$. 

Since $u_n(t,x)=v_n(t,x)$ thus converges to both $u(t,x)$ and $v(t,x)$, 
it follows that $u(t,x)=v(t,x)$.
\end{proof}

\begin{proof}[Proof of Theorem \ref{main2}]
  It is easy to see that the barrier \eqref{Bfromv} is standard.

If two standard barriers have the same survival distribution function $g$, then they have
the same function $v(x,t)$ by \eqref{v}, and thus the same function $u(x,t)$
by Theorem \ref{main3}. Proposition \ref{recoveryfromu} now implies that the
two barriers coincide.
\end{proof}

\section{An optimal stopping time}\label{S:stopp}

\begin{proof}[Proof of Theorem \ref{Tstopp}]
  The discrete case was shown in Corollary \ref{CX}.

For a general $g$, fix $t>0$ and
consider the discrete barriers $B_n$ constructed in
Section \ref{S:uniqueness}, now adding the point $t$ to $A_n(b)$
Consider also the corresponding boundaries $b_n(s)$ and
reflected boundaries $\bbt_n(s)$, and the stopping times $\ggtxn$ given by
\eqref{ggtx} with $\bbt_n(s)$.
Then $\ggtxn\ge \ggtx$, and thus
\begin{equation}
  g^n(t-\ggtxn)
\ge g^n(t-\ggtx)\ge g(t-\ggtx).
\end{equation}
Furthermore, in analogy with Lemma \ref{LP4},
$\ggtxn\to\ggtx$ a.s.; in fact the proof of Lemma \ref{LP4} holds without
changes when $\ggtx>0$, and if $\ggtx=0$, then $x\ge b(t)$, and because
we now have $t\in A_n(b)$, $b_n(t)=b(t)\le x$ so $\ggtxn=0$ for every $n$.
In particular, if $\ggtx<t$, then a.s.\ $\ggtxn<t$ for all large $n$.
Hence, a.s.,
\begin{multline*}
\liminf_\ntoo
\left[g^n(t-\ggtxn)\etta_{\{\ggtxn<t\}}+ \etta_{\{\ggtxn=t, x+W_t\geq
	0\}}\right]
\\
\ge 
\left[g(t-\ggtx)\etta_{\{\ggtx<t\}}+ \etta_{\{\ggtx=t, x+W_t\geq
	0\}}\right]
\end{multline*}
and Fatou's lemma yields
\begin{multline}\label{fatou}
\liminf_\ntoo\E
\left[g^n(t-\ggtxn)\etta_{\{\ggtxn<t\}}+ \etta_{\{\ggtxn=t, x+W_t\geq
	0\}}\right]
\\
\ge \E
\left[g(t-\ggtx)\etta_{\{\ggtx<t\}}+ \etta_{\{\ggtx=t, x+W_t\geq
	0\}}\right]
\end{multline}
The left-hand side of \eqref{fatou} 
is $\liminf_\ntoo v_n(t,x)$ by the discrete case,
and we have shown $v_n(t,x)\to v(t,x)$ in the proof of Theorem \ref{main3}.
Hence,
\begin{equation}
\E
\left[g(t-\ggtx)\etta_{\{\ggtx<t\}}+ \etta_{\{\ggtx=t, x+W_t\geq
	0\}}\right]
\le v(t,x),
\end{equation}
which shows that $\ggtx$ yields the infimum in \eqref{v}, i.e., \eqref{vopt} holds.
\end{proof}

\section{Continuity}
\label{regularity}

It is easy to see that if the boundary $b$ is continuous, then
the survival distribution function $g$ is continuous, but not conversely.
More precisely, we have the following result which says that $b$ may jump up
but not down for a continuous $g$.

\begin{theorem}
  \label{Tgcont}
Let $B$ be a standard barrier. Then the following are equivalent.
\begin{romenumerate}
\item \label{tcontg}
The survival distribution function
$g$ is continuous.

\item \label{tcontt}
For any fixed $t>0$, $\P(\tau_B=t)=0$.

\item \label{tcontliminf}
For every $t>0$,
\begin{equation}
\label{liminf}
\liminf_{s\upto t}b(s)=b(t)
\end{equation} 

\item \label{tcontw}
$W_{\tau_B}=b(\tau_B)$ a.s.\ when $\tau_B<\infty$.
\end{romenumerate}
\end{theorem}

\begin{proof}
\ref{tcontg}$\iff$\ref{tcontt}.
Obvious, since $\P(\tau_B=t)=g(t-)-g(t)$.

\ref{tcontt}$\implies$\ref{tcontliminf}.  
Recall that we know that $\liminf_{s\upto t}b(s)\geq b(t)$ since $b$
is lower semi-continuous.
To reach a contradiction, assume that \eqref{liminf} is not satisfied,
i.e. that there exists $t>0$ such that 
\[\liminf_{s\upto t} b(s)> b(t).\]
Then, by a similar argument as in the proof of
Proposition~\ref{recoveryfromu} one can show that 
$ \P(\tau_B=t)>0$, which contradicts \ref{tcontt}.
Thus \eqref{liminf} holds.

\ref{tcontliminf}$\implies$\ref{tcontw}.  
Since $W_{s}<b(s)$ for $s<\tau_B$, \eqref{liminf} and the continuity of $W$
imply 
$W_{\tau_B}=\lim_{s\upto \tau_B}W_s\le b(\tau_B)$ when $\tau_B<\infty$.
Furthermore, $W_{\tau_B}\ge b(\tau_B)$, and the result follows.

\ref{tcontw}$\implies$\ref{tcontt}.  
If \ref{tcontw} holds, then $\P(\tau_B=t)\le \P(W_t=b(t))=0$ for any  $t>0$.
\end{proof}

The converse, to find conditions on $g$ that correspond to continuous $b$
seems
more difficult. We do not know any necessary and sufficient condition, but
the next theorem yields a simple sufficient condition.

As a preparation, note that the function $g(t)$ is monotone, and thus 
has a derivative $g'(t)$ a.e., necessarily with $g'(t)\le0$.
Note that, for any interval $(a,b)\subseteq(0,\infty)$,
\begin{equation}\label{llip0}
 g(a)-g(b)=\int_a^b -g(dt) \ge \int_a^b -g'(t)\,dt, 
\end{equation}
with equality if the positive measure $-g(dt)$ is absolutely continuous, but
in general there may be strict inequality. Nevertheless, it is easy to see
from \eqref{llip0} and the definition of the derivative 
that for any interval $(t_0,t_1)$ and any constant $C>0$,
\begin{equation}
  \label{llip1}
-g'(t)\ge C  \text{ for a.e.\ } t\in(t_0,t_1)
\end{equation}
 if and only if
 \begin{equation}
   \label{llip2}
g(a)-g(b)\ge C(b-a)
 \end{equation}
for every interval $(a,b)\subseteq (t_0,t_1)$.

\begin{theorem}\label{Tbcont}
If the survival distribution function $g(t)$ is continuous and 
satisfies \eqref{llip1}, or equivalently \eqref{llip2},  for an interval
$(t_0,t_1)\subseteq (0,\infty)$ and some $C>0$,
then
the solution $b$ of the corresponding inverse first-passage problem is
continuous on $(t_0,t_1)$. 
\end{theorem}

\begin{proof}
Assume that $b$ is a solution of the inverse first-passage problem so that 
$\tau:=\inf\{t>0:W_t\geq b(t)\}$ satisfies $\P(\tau>t)=g(t)$.
By Theorem \ref{Tgcont}, \eqref{liminf} holds for $t\in(t_0,t_1)$.
(The proof requires only that $g(t)$ is continuous for $t\in(t_0,t_1)$.)
We next claim that 
\begin{equation}
\label{right}
\liminf_{s\downto t} b(s)= b(t)
\end{equation}
for $t\in(t_0,t_1)$.
Assume, to reach a contradiction, that \eqref{right} does not hold. Since $b$ is lower semi-continuous, this means that 
$\liminf_{s\downto t} b(s)> b(t)$ for some $t$. Thus there exists $\ep>0$ such that 
$b(t+s)\geq b(t)+\ep$ for all $s\in(0,\ep)$. For such $s$ we have 
\begin{align*}
g(t)-g(t+s)
&=
\P(\tau\in(t,t+s]) \leq \P\Bigpar{W_t<b(t),\sup_{t< r\leq t+s}W_r>b(t)+\ep}\\
&\leq \P\Bigpar{\sup_{t\leq r\leq t+s}(W_r-W_t)>\ep}
=\P\Bigpar{\sup_{0\leq r\leq s}W_r>\ep}.
\end{align*}
Since 
\[
\P\Bigpar{\sup_{0\leq r\leq s}W_r>\ep}= \sqrt{\frac{2}{\pi}}\int^\infty_{\ep/\sqrt{s}} e^{-z^2/2}dz\leq 
\sqrt{\frac{2s}{\pi \ep^2}}e^{-\ep^2/(2s)},\]
we find using \eqref{llip2} that 
\[0<C\leq \frac{g(t)-g(t+s)}{s}\leq \sqrt{\frac{2}{\pi s\ep^2}}e^{-\ep^2/(2s)}\to 0\]
as $s\to 0$. This contradiction shows that \eqref{right} holds.

In view of \eqref{liminf} and \eqref{right}, to prove continuity of $b$ it
suffices to show that  
\begin{equation}
\label{limsup}
\limsup_{s\to t} b(s)\le b(t)
\end{equation}
for any $t\in(t_0,t_1)$.
To do that, let $t_2\in[t_0,t_1)$, $a\in(0,t_1-t_2)$, $d>0$, and let 
\[E:=\{t\in(t_2,t_2+a): b(t)>b(t_2)+d\}.\]
Suppose that $E\not=\emptyset$. Since $b$ is lower semi-continuous, $E$ is open. Thus $\lambda(E)>0$, where
$\lambda$ denotes Lebesgue measure. Furthermore, 
by \eqref{llip1},
\[\P(\tau\in E)\ge \int_E -g'(t)\,dt\geq C\lambda(E).\]
(The first inequality may be strict if $g(dt)$ has a singular component.) 
Replace the boundary $b$ by 
\[b_1(t):=
\begin{cases}
b(t), & t\leq t_2\\
b(t_2)+d, & t>t_2,  
\end{cases}
\]
and let 
\[\tau_1:=\inf\{t>0:W_t\geq b_1(t)\}\]
be the corresponding stopping time.

If $\tau\in E$, then $W(\tau)=b(\tau)>b(t_2)+d$ and $W(t_2)<b(t_2)$. Thus $W(t)$ hits $b(t_2)+d$ somewhere in $(t_2,\tau)$, and since
$W(t)$ has not hit $b(t)$ in $[0,t_2]$, we have $t_2<\tau_1<\tau$. Hence $b(t_2)+d=W(\tau_1)<b(\tau_1)$, which shows that
$\tau_1\in E$. Thus if $\tau\in E$, then also $\tau_1\in E$.
Consequently,
\begin{equation}
\label{lowbound}
\P(\tau_1\in E)\geq \P(\tau\in E)\geq C\lambda(E).
\end{equation}

Next, condition on $\{\tau>t_2\}$ (equivalently, on $\{\tau_1>t_2\}$) and on $\{W(t_2)=y\}$, with $y<b(t_2)$.
Let $h=b(t_2)+d-y>d$. Then $\tau_1-t_2 \,{\buildrel d \over =}\,\inf\{t\geq 0:W(t)=h\}$, which has the density
\[\varphi_h(t)=\frac{h}{\sqrt{2\pi t^3}}e^{-h^2/(2t)}
\leq \frac{C_1}{h^2}e^{-h^2/(3t)}\leq \frac{C_1}{d^2}e^{-d^2/(3t)}.\]
Since the right-hand side does not depend on $y$, 
$\tau_1$ has a density $f_1(t)$ on $(t_2,\infty)$ which satisfies 
\[f_1(t)\leq \frac{C_1}{d^2}e^{-d^2/(3(t-t_2))}
\P(\tau >t_2)\leq \frac{C_1}{d^2}e^{-d^2/(3(t-t_2))}.\]
Since $E\subset (t_2,t_2+a)$, we obtain 
\begin{equation}
\label{upbound}
\P(\tau_1\in E)=\int_E f_1(t)\,dt\leq \frac{C_1}{d^2}e^{-d^2/(3a)}\lambda(E).
\end{equation}
Combining \eqref{lowbound} and \eqref{upbound} yields
\[C\lambda(E)\leq \frac{C_1}{d^2}e^{-d^2/(3a)}\lambda(E),\]
and thus, recalling $\lambda(E)>0$, 
\begin{equation}\label{dC}
d^2e^{d^2/(3a)}\leq C_1/C.
\end{equation}
We assume for simplicity, and without loss of generality, that $t_1-t_0\le
e^{-1}$. Then $a< t_1-t_2<e^{-1}$.
By \eqref{dC},
either $d^2\leq aC_1/C$ or $e^{d^2/(3a)}\leq a^{-1}$,
and in both cases,
\begin{equation}
\label{boundond}
d\leq C_2\sqrt{a\ln (1/a)}.
\end{equation}
Recalling the definition of $E$, we have shown that if $t\in(t_2,t_2+a)$ and $b(t)>b(t_2) + d$, then 
\eqref{boundond} holds. Consequently, for every $t\in(t_2,t_2+a)$, 
\[b(t)-b(t_2)\leq C_2\sqrt{a\ln(1/a)}.\]
Furthermore, for every $t\in(t_2,t_1)$, this holds for every $a\in(t-t_2,t_1-t_2)$, and thus
\[b(t)-b(t_2)\leq C_2\sqrt{(t-t_2)\ln(1/(t-t_2))} .\]
Note that the constant $C_2$ only depends on the lower bound $C$ in
\eqref{llip1}, so 
\begin{equation}
\label{holder}
b(t)-b(s)\leq C_2\sqrt{(t-s)\ln(1/(t-s))}
\end{equation}
for any $t_0<s<t<t_1$.
This immediately implies
\begin{equation}
\label{limsupdown}
\limsup_{u\downto t} b(u)\le  b(t)
\end{equation}
for $t\in(t_0,t_1)$.
Furthermore, if $t_0<s<t<t_1$, apply \eqref{holder} with $t$ replaced by
$u\in(s,t)$ and take the $\limsup$ as $u\upto t$. This yields
\begin{equation}\label{sw}
  \limsup_{u\upto t} b(u)\le b(s)+C_2\sqrt{(t-s)\ln(1/(t-s))}.
\end{equation}
Now take the $\liminf$ of \eqref{sw} as $s\upto t$. This yields, using
\eqref{liminf}, 
\begin{equation}\label{ring}
  \limsup_{u\upto t} b(u)\le \liminf_{s\upto t}b(s)=b(t).
\end{equation}
Combining \eqref{limsupdown} and \eqref{ring} we obtain 
\eqref{limsup},
which completes the proof.
\end{proof}

\begin{remark}
Note that the above proof (with $t_2=t_0=0$) 
also shows that if $-g^\prime\geq C>0$ on $(0,t_1)$, then 
setting $b(0)=0$ makes $b$ continuous on $[0,t_1)$.
\end{remark}

\begin{remark}
We have shown the one-sided estimate \eqref{holder}, which implies a one-sided H\"older($\gamma$) estimate for any $\gamma<1/2$
in any interval where $-g^\prime$ is bounded below. 
We do not know whether a corresponding upper bound on the density would guarantee a similar estimate also of $b(s)-b(t)$. Since 
we only need continuity of the boundary in the arguments in the next
section, we refrain from a more detailed study of regularity. 
\end{remark}

\section{A non-linear integral equation for the boundary}
\label{S:inteqn}

In this section we show that if a given survival distribution function $g$ has a positive density, then the corresponding boundary $b$ 
can be characterised as the unique solution of a certain non-linear integral equation. The techniques are a mixture of techniques
from the first-passage time problem and from optimal stopping problems.

The integral equation associated with the boundary $b$ is (cf. \cite[Section 14]{PS})
\begin{equation}
\label{inteqn}
\Psi\left( \frac{b(t)}{\sqrt{t}}\right)= \int_0^t\Psi\left(\frac{b(t)-b(s)}{\sqrt{t-s}}\right)(1- g)(ds)
\end{equation}
for $t>0$, where $\Psi=1-\Phi$ and 
\[\Phi(y)=\frac{1}{\sqrt{2\pi}}\int^y_{-\infty} e^{-\frac{z^2}{2}}dz\]
is the distribution function of the standard normal distribution.

\begin{remark}
The standard way in optimal stopping theory to derive an integral equation for the boundary is to 
apply It\^o's formula to the process $v(t-s,x+W_{s})$, to argue that the local martingale term arising in 
It\^o's formula is in fact a martingale, and then to insert $x=b(t)$ and take expected values.
To apply this scheme, one needs to assume that $v$ is regular enough to apply It\^o's formula (in particular, 
also over the boundary and up to the terminal date). The applicability of It\^o's formula often relies on 
arguments using the monotonicity of the optimal stopping boundary and the smooth fit condition. 
In the present setting, however, the boundary is typically not monotone, and the smooth fit condition is
not expected to hold at all boundary points. Instead, we derive \eqref{inteqn} below using methods from 
first-passage time problems. As a contrast, the uniqueness statement is proved using methods from
optimal stopping theory.
\end{remark}

\begin{theorem}
\label{contsurv}
Assume that $g$ is a continuous survival distribution function. Then any solution $b$ of the inverse first-passage problem
satisfies \eqref{inteqn}.
\end{theorem}

\begin{proof}
By Theorem \ref{Tgcont},
$W_{\tau_B}=b(\tau_B)$. Therefore, similar calculations as in the proof of 
\cite[Theorem 14.2]{PS} yield
\begin{eqnarray}
\label{palme}
\Psi\left(\frac{b(t)}{\sqrt{t}}\right)&=& \P(W_t\geq b(t))\\
\notag
&=& \P(W_t\geq b(t),\tau_B\leq t)\\
\notag
&=&\int_0^t\P(W_t\geq b(t)\mid\tau_B=s)(1-g)(ds)\\
\notag
&=&  \int_0^t\Psi\left(\frac{b(t)-b(s)}{\sqrt{t-s}}\right)(1-g)(ds)
\end{eqnarray}
for $t>0$, which finishes the proof.
\end{proof}

\begin{theorem}
\label{uniqueness}
Let $0<T\le\infty$.
If the survival distribution function $g$ is continuous and satisfies \eqref{llip1}, or
equivalently \eqref{llip2},
with some constant $C(t_0,t_1)>0$ on  
$(t_0,t_1)$ for every $t_0,t_1$ with $0<t_1<t_2<T$, 
then $b$ is the unique continuous solution of \eqref{inteqn} on $(0,T)$
with 
\begin{equation}\label{tu}
\liminf_{t\downto 0}b(t)\geq 0.
\end{equation}
\end{theorem}

\begin{remark}
The proof of Theorem~\ref{uniqueness} follows the proof of Theorem 4.3 in \cite{J} and Remark 3.2 in \cite{P}.
\end{remark}

\begin{proof}
By Theorem~\ref{contsurv}, $b$ solves \eqref{inteqn}; furthermore, $b$ is continuous on $(0,T)$
by Theorem~\ref{Tbcont} and \eqref{tu} holds since otherwise
$\tau_B=0$ a.s.\ and $g(t)=0$, contradicting our assumptions.

For uniqueness, assume that $c:(0,T)\to\R$ is continuous, that 
\begin{equation}
\label{inteqnc}
\Psi\left( \frac{c(t)}{\sqrt{t}}\right)= \int_0^t\Psi\left(\frac{c(t)-c(s)}{\sqrt{t-s}}\right)(1- g)(ds)
\end{equation}
for $t\in(0,T)$ and that
\begin{equation}
\label{liminf0}
\liminf_{t\downto 0}c(t)\geq 0.
\end{equation}
For $(t,x)\in[0,T)\times\R$, define 
\begin{equation}
\label{V}
V(t,x):= \P(x+W_t\geq 0)+\int_0^t\P\left(x+W_{t-u}\geq c(u)\right)g(du)
\end{equation}
and note that $V(0,x)=\etta_{[0,\infty)}(x)$ and, as a consequence of
  \eqref{inteqnc}, for $t\in(0,T)$,  
  \begin{equation}\label{Vc=g}
	\begin{split}
V(t,c(t))&
= 1-\Psi\left( \frac{c(t)}{\sqrt{t}}\right)
+ \int_0^t\left(1-\Psi\left(\frac{c(t)-c(s)}{\sqrt{t-s}}\right)\right)g(ds)
\\&= 1
+ \int_0^t g(ds)
=
g(t). 	  
	\end{split}
  \end{equation}
Moreover, it follows from the Markov property that the process 
\begin{equation}
\label{Ms}
M_s:=V(t-s,x+W_s)+\int_{t-s}^t I(x+W_{t-u}\geq  c(u)) g(du)
\end{equation}
is a martingale. 
Note that $M_0=V(t,x)$.

Next, note that a similar calculation as in \eqref{palme} yields
\begin{eqnarray*}
u(t,x) &=& \P(W_t\leq x,\tau>t)\\
&=& \P(W_t\leq x) -\P(W_t\leq x,\tau\leq t)\\
&=& \Phi\left(\frac{x}{\sqrt t}\right)-\int_0^t\Phi\left(\frac{x-b(s)}{\sqrt{t-s}}\right)(1-g)(ds)
\end{eqnarray*}
for $t\in(0,T)$, so by Theorem~\ref{main3} we have, cf.\ \eqref{V},
\begin{eqnarray}
\label{vrepr}
v(t,x) &=& \Phi\left(\frac{x}{\sqrt t}\right)-\int_0^t\Phi\left(\frac{x-b(s)}{\sqrt{t-s}}\right)(1-g)(ds)\\
\notag
&=&
\P(x+W_t\geq 0)+\int_0^t\P\left(x+W_{t-u}\geq b(u)\right) g(du)
\end{eqnarray}
for $t\in(0,T)$.
Therefore, by the Markov property, also the process
\begin{equation}
  \label{ms}
m_s:=v(t-s,x+W_s)+\int_{t-s}^t I(x+W_{t-u}\geq  b(u)) g(du)
\end{equation}
is a martingale. Note that $m_0=v(t,x)$.

\smallskip\noindent
\emph{Claim 1:} $V(t,x)=g(t)$ for $(t,x)\in(0,T)\times\R$ with $x\geq c(t)$.

For $x=c(t)$, this is shown in \eqref{Vc=g}. 
Hence, assume $x> c(t)$, and define the stopping time
\begin{equation}
  \label{gammac}
\gamma_{c}:=\inf\{s\in[0,t]:x+W_s\leq  c(t-s)\}\wedge t.
\end{equation}
If $\gamma_c<t$, then $x+W_{\gamma_c}=c(t-\gamma_c)$ by \eqref{gammac} and
continuity, and thus 
\begin{equation}
  \label{vtc}
V(t-\gamma_c,x+W_{\gamma_c})=g(t-\gamma_c)
\end{equation}
by \eqref{Vc=g}. If $\gamma_c=t$, then $x+W_s>c(t-s)$ for $0<s\le t$ by
\eqref{gammac}, and letting $s\upto t$ yields, using \eqref{liminf0}, 
$x+W_t\ge 0$; hence $V(0,x+W_{t})=1=g(0)$ and \eqref{vtc} holds in
this case too.
Furthermore (in both cases), $x+W_s>c(t-s)$ for $s<\gamma_c$, 
and thus \eqref{Ms} yields, using \eqref{vtc}, 
\begin{eqnarray*}
M_{\gamma_{c}} &=& V(t-\gamma_{c},x+W_{\gamma_{ c}}) + \int_{t-\gamma_c}^t g(du)\\
&=& g(t-\gamma_{c}) + g(t)-  g(t-\gamma_{ c})=  g(t).
\end{eqnarray*}
Consequently, $V(t,x)=M_0=\E[M_{\gamma_{ c}}]=g(t)$.

\smallskip\noindent
\emph{Claim 2:} $V(t,x)\geq v(t,x)$ for $(t,x)\in[0,T)\times\R$.

For $t=0$ we have equality, and for $(t,x)\in(0,T)\times \R$ with $x\geq c(t)$, the claim follows from Claim 1 and the fact that $v(t,x)\leq g(t)$.
For $(t,x)\in(0,T)\times \R$ with $x< c(t)$, define the stopping time
\[\tau_{c}:=\inf \{s\in[0,t]:x+W_s\geq  c(t-s)\}\wedge t.\]
For $s\leq\tau_c$ and $t-s<u<t$, we have $t-u<s\le\tau_c$ and thus
$x+W_{t-u}<c(u)$. Hence, \eqref{Ms} yields 
$M_s=V(t-s,x+W_s)$ for $s\leq \tau_c$ and thus, using \eqref{Vc=g} when
$\tau_c<t$ and thus $x+W_{\tau_c}=c(t-\tau_c)$,
\begin{align*}
V(t,x) &= 
M_0=\E M_{\tau_c}
=
\E\left[V(t-\tau_c,x+W_{\tau_c})\right]\\
&=  \E\left[g(t-\tau_c)\etta_{\{\tau_c<t\}} + \etta_{\{x+W_t\geq 0,\tau_c=t\}}  \right]\\
&\geq  v(t,x)
\end{align*}
by the definition \eqref{v} of $v$, which finishes the proof of Claim 2.

\smallskip\noindent
\emph{Claim 3:} $c(t)\leq b(t)$ for $t\in(0,T)$.

Assume, to reach a contradiction, that $c(t)>b(t)$ for some $t\in(0,T)$.
By continuity of $b$ and $c$, we may find $\ep >0$ such that 
$c(s)>b(s)+\ep$ for all $s\in [t-\ep,t]$.
Choose $x=c(t)$, and define the stopping time
\[\gamma_{b}:=\inf\{s\in[0,t]:x+W_s\leq b(t-s)\}\wedge \ep.\]
Then $V(t,x)=g(t)$ by \eqref{Vc=g} and $v(t,x)=u(t,x)=g(t)$ since $x=c(t)\ge
b(t)$. Hence, using \eqref{Ms}, \eqref{ms} and $b(u)<c(u)$ for
$u\in[t-\gamma_b,t]$, 
\begin{align*}
0 &= V(t,x)-v(t,x)
=M_0-m_0
= \E[M_{\gamma_b}-m_{\gamma_b}]\\
&= 
\E\left[V(t-\gamma_{b},x+W_{\gamma_b})-v(t-\gamma_{b},x+W_{\gamma_{b}})\right]
\\&
\qquad-\E \left[\int_{t-\gamma_b}^t I(b(u)\leq x+W_{t-u}<c(u))g(du)\right]
\\&
\ge-\int_{t-\eps}^t \P\bigpar{b(u)\leq x+W_{t-u}<c(u),\, \gamma_b\ge t-u}g(du)
>0
\end{align*}
since $V\geq v$ by Claim 2 and $g(du)$ is a negative measure that does not
vanish on $(t-\eps,t)$.
This contradiction proves the claim.

\smallskip\noindent
\emph{Claim 4:} $V(t,x)= v(t,x)$ for $(t,x)\in[0,T)\times\R$.

It follows from Claim 3 that 
$\P(x+W_{t-u}\ge c(u))\ge \P(x+W_{t-u}\ge b(u))$ for every $u\le t$, and thus
\eqref{V} and \eqref{vrepr} yields $V(x,t)\leq v(x,t)$, 
again recalling that $g(du)$ is a negative measure.
Hence $V=v$ by Claim 2.

\smallskip\noindent
\emph{Claim 5:} $c(t)\geq b(t)$ for $t\in(0,T)$.

Let $x=c(t)$. Then, by Claims 4 and 1, $v(t,x)=V(t,x)=g(t)$. Thus, by
\eqref{Bfromv} and Theorem \ref{main2}, $(t,x)\in B$ so $c(t)=x\ge b(t)$.

\smallskip
By Claims 3 and 5, $c(t)=b(t)$ for $t\in(0,T)$.
\end{proof}

\end{document}